\documentclass[11pt,a4paper, oneside]{article}
\usepackage[utf8]{inputenc}
\usepackage[T1]{fontenc}
\usepackage{lmodern}
\usepackage[ngerman, english]{babel}
\usepackage{amsmath}
\usepackage{amssymb}
\usepackage{amsthm}
\usepackage{graphicx}
\usepackage{tikz-cd} 
\usepackage{cleveref}
\usepackage[shortlabels]{enumitem}
\usepackage[Option]{overpic}
\graphicspath{{figures/}} 

\let\oldbibliography\thebibliography
\renewcommand{\thebibliography}[1]{%
  \oldbibliography{#1}%
  \setlength{\itemsep}{-1.2mm}%
}

\def\R{\mathbb{R}}

\def\cK{\mathcal{K}}

\def\cM{\mathcal{M}}

\def\txtd{{\textnormal{d}}}
\def\txte{{\textnormal{e}}}

\def\txts{{\textnormal{s}}}
\def\txtu{{\textnormal{u}}}

\def\txtD{{\textnormal{D}}}

\newcommand{\myendex}{$\blacklozenge$\end{ex}}
\newcommand{\myendexerc}{$\lozenge$\end{exerc}}
\newcommand{\myendpexerc}{$\lozenge$\end{pexerc}}

\numberwithin{equation}{section}
\newtheorem{theorem}{Theorem}[section]

\newtheorem{lemma}[theorem]{Lemma}
\newtheorem{proposition}[theorem]{Proposition}
\newtheorem{corollary}[theorem]{Corollary}

\newcommand{\rmd}{\mathrm{d}}
\newcommand{\rmD}{\mathrm{D}}
\newcommand{\nwc}{\newcommand}
\nwc{\red}[1]{\textcolor{red}{#1}}
\nwc{\blue}[1]{\textcolor{blue}{#1}}
\nwc{\green}[1]{\textcolor{green}{#1}}
\nwc{\ddt}{\frac{\text{d}}{\text{dt}}}
\nwc{\dds}{\frac{\text{d}}{\text{ds}}}

\let\epsilon=\varepsilon
\nwc{\B}{\mathcal{B}}
\nwc{\F}{\mathcal{F}}
\nwc{\G}{\mathcal{G}}

\author{ Luca Arcidiacono\thanks{Department of Mathematics, Technical University of
		Munich, Boltzmannstr.~3, Garching bei München, \mbox{D-85748}, Germany}~~
	and~~Christian Kuehn\footnotemark[1]
}

\title{Blowing-up Nonautonomous Vector Fields:\\ Infinite Delay Equations and Invariant Manifolds}

\date{\today}

\begin{document}
\maketitle

\begin{abstract}
\noindent We show the existence of nonautonomous invariant manifolds for planar, asymptotically autonomous differential equations, that have equilibrium solutions with zero Lyapunov spectrum. These invariant manifolds correspond to the stable and unstable manifold of a desingularized equation, that we obtain by using the blow-up method. More precisely, the blow-up method is extended to the nonautonomous setting and transforms the original finite-dimensional ordinary differential equation into an infinite-dimensional delay equation with infinite delay, but hyperbolic structure. In the technical construction of the invariant manifolds for the delay equation, we have to carefully study the effect of the time reparametrization used for desingularization in the blown-up space to guarantee sufficient regularity. This allows us to employ a Lyapunov-Perron argument to obtain existence of an invariant manifold. We combine the last step with an implicit function argument to study differentiability of the manifold. Finally, we reverse the blow-up transformation of space and time variables obtaining invariant manifolds of the initially considered equation.
\end{abstract}

\section{Introduction}
\label{sec:introduction}

We consider an nonautonomous planar ordinary differential equation (ODE) of the form
\begin{equation}
\label{eq:intro1}
	\frac{\rmd x}{\rmd t}=x'=F(x) + G(x)h(t)\,,\qquad x=x(t)\in\R^2, 
\end{equation}
with \(F, G: \R^2 \to \R^2\) and a function \(h:\R \to \R\) that decays exponentially towards zero, i.e., the system is asymptotically autonomous. We assume that the origin is an equilibrium solution with Lyapunov spectrum consisting only of zero; in particular, the equilibrium of the limit-equation \(x' = F(x) \) is  not hyperbolic.  Due to the asymptotic behaviour of equation~\eqref{eq:intro1} the Lyapunov spectrum coincides with the real part of the spectrum of the linearization of \(F\) at the origin, i.e., we assume that the eigenvalues of the Jacobian \(\rmD F(0)\) lie on the imaginary axis. 

For autonomous differential equations the blowup technique has turned out to be a powerful method to desingularize vector fields around non-hyperbolic equilibria. Originating in algebraic geometry to analyze singularities of algebraic varieties~\cite{Hironaka1,Hironaka2}, this technique was later on transferred and extended to autonomous ODEs; see \cite{Dumortier1} for an introduction to the method. A main classical result in the area is that one can always desingularize planar autonomous ODEs with non-hyperbolic equilibria~\cite{Dum77,Dumortier}. This result has been extended to autonomous ODEs in three dimension~\cite{Panazzolo2} relatively recently. Yet, even in higher-dimensional situations, the blowup is extremely helpful in many cases, e.g., for studying the loss of normal hyperbolicity in fast-slow dynamical systems~\cite{DumortierRoussarie,Krupa2001}; see also~\cite{KuehnBook}. The blowup method uses a non-injective transformation that maps a higher-dimensional object, such as a sphere or a line, onto the non-hyperbolic equilibrium constituting the singularity. The dynamics on this larger, blown-up version of the singularity may then be desingularized and exhibit (partially) hyperbolic behaviour. The desingularization step involves a suitable time re-parametrization depending on the degeneracy of the vector field. For autonomous ODEs, this time reparametrization does not pose a problem. However, it seems at first sight to prevent the use of blow-up techniques for general classes of nonautonomous dynamical systems~\cite{KloedenRasmussen,Poetzsche2}. 

Here we use a different viewpoint to prove that geometric desingularization via blow-up is still possible. We observe that applying a time reparametrization to equation~\eqref{eq:intro1} does not lead to another nonautonomous differential equation but to a delay differential equation~\cite{Hale3,Diekmannetal}. This provides a link between two previously disparate areas: the geometric desingularization of singularities for finite-dimensional dynamical systems and the theory of infinite-dimensional delay/functional differential equations. To be more precise on the structure of the delay, we are going to see that the derivative of the solution of the desingularized equation at any time \(t\geq t_0\) depends on the whole history of the trajectory on the interval \([t_0, t]\), where \(t_0\) is the fixed point in time that stays invariant under the time change. Hence, the delay becomes infinitely large as \(t\to \infty\). Furthermore, the vector field obtained after the desingularization is singular in the spatial limit \(x \to 0\), so that assumptions on \(h\) are necessary to extend it properly to the pre-image of the equilibrium point under the blowup map. One may then check that the desingularized delay equation now possesses a hyperbolic equilibrium solution, yet the singularity as \(x \to 0\) makes it more complicated to analyze. To show the true utility of the blow-up, we prove the existence of stable and unstable manifolds in the vicinity of the equilibrium for the delay equation. A classical technique to show existence of stable and unstable manifold is the Lyapunov-Perron method, that defines the manifold as solution sets of a fixed point integral equation. It was already used for autonomous~\cite{Perron1929}, nonautonomous~\cite{Morel2008,KloedenRasmussen} or delay~\cite{Mura07} equations and is also the method of choice for our situation. To define the Lyapunov Perron operator we utilize a variation of constants formula presented in \cite{Hino2002} that covers delay equations with infinite delay. The resulting stable/unstable manifolds correspond, after a suitable time evaluation and blow-down, to invariant manifolds of the originally considered equation~\eqref{eq:intro1}. This allows to draw conclusions about the dynamics in the vicinity of a non-hyperbolic singularity (in the sense of Lyapunov exponents), which completes the key rigorous proof-of-principle that blow-up techniques can be applied to nonautonomous ODEs. \medskip 

Let us briefly outline the structure of paper. We start by giving a short introduction to the blowup method as well as time reparametrization and state some basic concepts of stability for linear and nonlinear autonomous and nonautonomous differential equations in Section~\ref{sec:background}. Next, in Section~\ref{sec:setting_main_result} we provide all relevant assumptions and definitions allowing us to formulate Theorem~\ref{thm:stable_mfld} on the existence of a stable manifold and its properties. Furthermore, we explain how these lead to invariant manifolds of the original equation. Section~\ref{sec:preliminaries} covers some technical tools that are used for the proof of Theorem~\ref{thm:stable_mfld}, which is presented in Section~\ref{sec:proof}. The first part of Section~\ref{sec:proof} is dedicated to the regularity properties of the considered vector field while the latter part then utilizes all the above mentioned ideas to show existence of the stable manifold. Finally, we give a short outlook, which generalizations of our nonautonomous blowup are possible next challenges.

\section{Background}\label{sec:background}

\subsection{Blowup of Autonomous Differential Equations}
\label{sec:blowup_autonomous}

The blowup method originated in algebraic geometry and was later on modified and extended to desingularize non-hyperbolic equilibria of ordinary differential equations; see e.g.~\cite{Dum77,Dumortier1,KuehnBook}. It uses a non-injective blowup transformation \(\Phi\colon \cM \mapsto \R^n \) on some manifold \(\cM\) that maps a higher-dimensional object, as for example a sphere, onto the equilibrium point. For illustration, let us consider the example
\begin{align}\label{eq:example1}
\begin{split}
	x'&=x^2-2xy \,,\\
	y'&=y^2-2xy \,.
	\end{split}
\end{align}
Here the origin is non-hyperbolic as its linearization completely vanishes. The easiest geometric example of a blowup map is to use polar coordinates, i.e., $\Phi : \mathbb{S}^1 \times I \to \R^2, ~ (\theta, r) \mapsto (r \cos(\theta), r \sin(\theta))$. This change of coordinates induces (via pullback of the vector field above under \(\Phi\))  a new flow given by 
\begin{align}\label{eq:example2}
\begin{split}
\theta'&= r\cdot a(r, \theta)\,, \\
r'&= r^2 \!\cdot b(r, \theta)\,,
\end{split}
\end{align}
where $a(r, \theta)$, $b(r, \theta)$ are explicitly computable, yet quite involved and contain several trigonometric functions. Note that there is still no movement on the set \(\mathbb{S}^1 \times \{0\} = \Phi^{-1}(\{(0,0)\})\) for the transformed ODE~\eqref{eq:example2}. It is clear that using polar coordinates, the two induced flows will be smoothly conjugate outside of the origin $(0,0)$. Yet, note that we have acquired a new degree of freedom as $\mathbb{S}^1$ is an entire circle of equilibria. In particular, we obtain an equivalent set of equations 
 \begin{align}
\label{eq:example3}
\begin{split}
 \theta'&= ~a(r, \theta)  \,, \\ 
 r'&=r\cdot b(r,\theta) \,,
 \end{split}
 \end{align}
which is  more regular,  by dividing out the common factor \(r^k\), see also Section~\ref{sec:time_reparametrization} for more details on this equivalence. In general, such a desingularized version \(\hat{F}(\theta, r)= \frac{1}{r^k}F(\theta, r)\) of the vector field can be given if its \(k\)-jet vanishes. The desingularized vector field extends smoothly onto the circle \(\mathbb{S}^1 \times \{0\}\) and shows new equilibrium points on that circle, which is still an invariant set. In contrast to the original vector field, the newly gained equilibria have increased hyperbolicity. Now typical linear theory may be applied to those equilibria and results can be transferred back to the original vector field by using \(\Phi\).
 
In the calculations of the blow-up it turns out to be helpful to avoid explicit trigonometric calculations. Instead, one uses directional blowup maps -- the pre-image of the origin is now one of the axes -- which can be seen as a concatenation of the previous polar version with a chart map/projection of the circle. It might be necessary to consider multiple directional blowup maps, just as it is necessary to use multiple charts to cover the circle. A directional blowup utilizes a coordinate transformation of the form \((x,y)=(u,uv)\). In view of our example \eqref{eq:example1} above this yields
\vspace{-0.2cm}
\begin{align*}
	u' &= u^2(1-2v)\,, \\
	v' &= 3uv(v-1) \,.
\end{align*}
Again a common factor may be eliminated to resolve the singular structure. The desingularized equation
\vspace{-0.2cm}
\begin{align*}
u' &= u(1-2v)\,, \\
v' &= 3v(v-1)\,,
\end{align*}
has two hyperbolic saddle points at \((u,v) = (0,0)\) and \((u,v)=(0,\frac{1}{2})\) located on the invariant line \(\{u=0\}\). The directional blowup will be the type of blowup that we use in the rest of this paper due to its benefits of clear, easily computable equations. \medskip

A generalization of those blowups is given by the quasi-homogeneous blowup for polynomial vector fields, that can often simplify calculations even further. One uses weights on the polar transformation in different spatial directions and maps an ellipsoid onto the equilibrium point.  

\subsection{Time Reparametrization}
\label{sec:time_reparametrization}

Heuristically, equivalence of equations like \eqref{eq:example2} and \eqref{eq:example3} seems reasonable since at any point in the phase space the vector field has the same direction but only different length. Thus trajectories following the vector field should share the same orbit. This can be made precise by introducing a time reparametrization. \medskip

 Let \(U\subset \R^n\) be an open, invariant set of the general autonomous ODE  
\begin{equation}
\label{eq:ode1}
x'=f(x).
\end{equation}
Since we are only interested in local dynamics near an equilibrium, we assume without loss of generality that~\eqref{eq:ode1} generates a global-in-time flow. Let \(z: U \to \R^+\) be some positive function and consider the equation
\begin{equation}
\label{eq:ode2}
x'=f(x)z(x)\,.
\end{equation}
 One may construct the mapping that takes trajectories of \eqref{eq:ode1} to those of \eqref{eq:ode2} as follows: Take a trajectory \(\xi : \R \to \R^n\) of \eqref{eq:ode1} with initial condition  \(\xi(0)= \xi_0 \in U\) and define the function 
\[ \omega(s) =  \int_{0}^{s} \frac{1}{z(\xi(r))} \,\txtd r .\] 
Since \(z\) is positive, \(\omega : \R \to \R\) is strictly increasing, hence invertible on its range \(K\). Let us denote its inverse by \(\rho: K \to \R\). This inverse function \(\rho\) satisfies the identity
\[ \rho(t) = \int_{0}^{t} z(\xi (\rho(u)))\, du \] 
so that we have \(\rho'(t) = z(\xi(\rho(t)))\). Note further that for any trajectory we have \(\rho(0)=0\), i.e. the initial time 0 remains the same. To obtain the equivalence mapping between the two systems \eqref{eq:ode1} and \eqref{eq:ode2} we now rescale the time paths using \(\rho\). The new trajectories \(\psi(t) := \xi(\rho(t))\) satisfy due to the chain rule
\[ \ddt \psi(t) = \xi'(\rho(t)) \, \rho'(t) = f(\psi) z(\psi)\,, \] 
showing that trajectories of \eqref{eq:ode1} get mapped onto those of \eqref{eq:ode2} by the trajectory-wise time rescaling \(\rho\). Hence the two systems \eqref{eq:ode1} and \eqref{eq:ode2} are topologically equivalent. \medskip

\textit{Remark:} The special case of a constant function \(z\equiv c\) yields a ``classical'' rescaling of the time and \(\rho(t)=c\cdot t\) for any trajectory. However, note that in general we obtain a different rescaling for every trajectory.\medskip 

Besides during the desingularization of non-hyperbolic equilibria, a rescaling of differential equations is also useful for proving topological properties of trajectories or to complete a flow, where the time reparametrization yields an accessible way of an equivalence mapping, see also \cite{Chicone2006}.

\subsection{Stability in Autonomous and Nonautonomous Systems}
\label{sec:stability_auto_nonauto}

To determine the stability of an equilibrium point of an autonomous differential equation it is often sufficient to consider the linearization of the equation, i.e. the variational equation around this steady solution. In case of a hyperbolic equilibrium point, when there are no eigenvalues of the linearization on the imaginary axis, there is a splitting of the phase space into stable and unstable eigenspaces in which solutions decay with exponential rates according to the eigenvalue in forward and backward time respectively~\cite{GH}. Corresponding objects to the eigenspaces of the linear equation carry over to the nonlinear equation as stable and unstable manifolds. Theses can again be characterized as the solutions that converge in forward and backward time to the equilibrium point. The existence of stable and unstable manifolds allows to construct a homeomorphism that leads to the classical Hartman-Grobman Theorem~\cite{Chicone2006}:\bigskip

\textit{Assume the differential equation \(x' = f(x)\) possesses a hyperbolic equilibrium point \(x^*\). Then the linear equation \(\,X'=\textnormal{D}f(x_*)X\,\) is locally topologically equivalent to the nonlinear equation \(x' = f(x)\). } \medskip

These results are a main motivation for employing a blow-up to gain hyperbolicity. In particular, in the blown-up space, one may apply tools from linear analysis and try to prove the existence of stable/unstable manifolds, and then eventually transfer the results back to the original vector field by a blow down.\bigskip 

For nonautonomous differential equations $ x' = f(x,t)$, the eigenvalues, even along a steady solution, can vary as the vector field changes in time. One suitable generalization of the eigenvalues are Lyapunov exponents and correspondingly the Lyapunov spectrum \cite{KloedenRasmussen}.  We again consider the variational equation \(X'= \textnormal{D}f(x^*,t)X\) with corresponding flow map \(\Phi\) and define for \(\xi_0 \in \R^n\)
\[ \lambda_+^{\infty} (\xi_0) = \limsup_{t \to \infty} \frac{1}{t} \ln \|\Phi(t,0) \xi_0\|  \qquad \text{ and } \qquad  \lambda_-^{\infty} (\xi_0) = \liminf_{t \to \infty} \frac{1}{t} \ln \|\Phi(t,0) \xi_0\|   \,.\]
The set
$$ \sigma_{Ly}^\infty = \bigcup_{\xi_0\neq 0} [\lambda_-^{\infty}, \lambda_+^{\infty}] $$
is called the Lyapunov spectrum. It coincides with the real part of the normal (eigenvalue-) spectrum in the case of an autonomous equation. Due to the limiting behaviour of equation~\eqref{eq:intro1} this is also the case for the equations considered in this work so the spectrum becomes asymptotically autonomous. Under the presence of an exponential splitting in stable and unstable directions (which may now vary in time) the existence of stable and unstable manifolds for nonautonomous differential equations can be shown, see e.g.~\cite{KloedenRasmussen}. However, the non-hyperbolic equilibrium point of equation~\eqref{eq:intro1} does not allow such a strategy. We are going to prove that a blowup transformation can yield an exponential estimate of the linearization, but does so at the cost of a more difficult delay equation, which is not a standard nonautonomous equation anymore.  

\section{Setting and Main Result}
\label{sec:setting_main_result}

In this paper we consider nonautonomous differential equations in \(\R^2\) of the  form
\begin{equation}\label{eq:main}
\frac{\txtd}{\txtd\tau}\begin{pmatrix}
	x\\ y
\end{pmatrix}=	\begin{pmatrix}
	x\\ y
	\end{pmatrix}' =F(x,y) +G(x,y) h(\tau)
\end{equation}
with sufficiently smooth  functions \(F,G\colon \R^2 \to \R^2 \)  and  an exponentially decaying function \({h \in C^1(\R ; \R)} \). To be more precise  let	\(h(\tau)\) converge to 0 as \(\tau \to \pm\infty\) and satisfy 
\begin{equation}\label{eq:assumption}
|h(\tau)| \leq H \txte^{-\eta \tau}  ~\text{ as well as }~  |h'(\tau)| \leq  H \txte^{-\eta \tau}
\end{equation}
for all \(\tau \in \R\) and some positive constants \(H, \eta \in \R\). One might also consider a function \(h\) that is decaying exponentially only to one end, i.e. suppose \eqref{eq:assumption} holds for all \(\tau\in \R^+\) or all \(\tau\in \R^-\). In these cases Theorem~\ref{thm:stable_mfld} can only be applied in one time direction to obtain either a stable or unstable manifold. Such a function \(h\) might for example be obtained as the solution of a scalar differential equation with a stable equilibrium. We further assume that \(F(0,0) = G(0,0) = (0,0)\) so that the origin is an equilibrium point. The origin is supposed to be non-hyperbolic in the sense that the Lyapunov spectrum contains the value \(0\). It is straightforward to verify that~\eqref{eq:assumption} implies that the Lyapunov spectrum consists exactly of the real parts of the eigenvalues of the Jacobian \(\textnormal{D}F(0)\). In particular, in the limits of \(\tau \to \pm\infty\) we have an asymptotically autonomous equation and we assume that \(\textnormal{D}F(0)\) has at least one eigenvalue on the imaginary axis.\medskip 

We apply a blowup coordinate transformation in combination with a time rescaling to desingularize our equation~\eqref{eq:main} at the origin. We perform a directional blowup 
\begin{equation*}
	(x,y)=(u,uv)
\end{equation*}  
that stretches the origin to the invariant line \(\{u=0\}\) and leads to an equation of the form 
\begin{equation}
\label{eq:main_uv}
\begin{pmatrix}u\\v\end{pmatrix}' = \tilde{f}(u,v) + \tilde{g}(u,v) h(\tau)	
\end{equation}
for some functions \(\tilde{f}, \tilde{g} \colon \R^2 \to \R^2\). Note that since \((x,y)=(0,0)\) is an equilibrium point of equation~\eqref{eq:main} the right hand side of~\eqref{eq:main_uv} vanishes whenever \(u=0\). \medskip

We combine this transformation with a time rescaling 
\begin{equation}\label{eq:rescaling1}
	\omega(\tau) =  \int_{0}^{\tau} \frac{1}{z(\xi(r))}~\txtd r \,,
\end{equation} as introduced in Section~\ref{sec:time_reparametrization},
now for some trajectory \(\xi(\tau)\) of equation~\eqref{eq:main_uv}.\medskip

 One may also choose a different time \(\tau_0\in \R\) that remains invariant under the time reparametrization~\(\omega\). The rescaling then takes the form \(\omega(\tau) =  \int_{\tau_0}^{\tau} \frac{1}{z(\xi(s))}\,\txtd s \,\); we only consider the case \(\tau_0 = 0\) here. \medskip 

 In order to desingularize equation~\eqref{eq:main_uv} we choose \[ z\colon \big(\R\!\setminus\!\{0\}\big) \times \R \to \R^+ \,,~~~~ z(u,v) = |u|^{-\kappa} \] 
 for some \(\kappa \in \mathbb{N}\), which constitutes the degree of desingularization. As we have already seen before, the inverse \(\rho\) of \(\omega\) satisfies  $ \rho(t)=\int_{0}^{t} z(\psi(s)) \,\txtd s $  with rescaled trajectories \(\psi(t) := \xi(\rho(t))\).   
  
  These new trajectories $\psi(t) = (u(t), v(t))$  now satisfy the differential equation
\begin{equation}
\label{eq:delay_general}
 \ddt \begin{pmatrix}u \\ v
 \end{pmatrix}= \begin{pmatrix}u \\ v
\end{pmatrix}' = f(u,v) + g(u,v) \, h\!\left(\int_{0}^{t}\frac{1}{|u(s)|^\kappa}~\txtd s\right) .
\end{equation}
where $$ f(u,v)= \tilde{f}(u,v)\cdot z(u,v) = \frac{1}{|u|^\kappa}\,\tilde{f}(u,v)  \text{ ~~~and~~~ }  g(u,v)= \tilde{g}(u,v) \cdot z(u,v)  = \frac{ 1}{|u|^\kappa}\, \tilde{g}(u,v) \,.$$
The time has been changed (trajectory-wise) according to  
\begin{equation}\label{eq:rescaling2}
	\tau = \rho(t)=  \int_{0}^{t} \frac{1}{|u(s)|^\kappa} \,\txtd s \,.
\end{equation}

 Note that the right hand side of  equation \eqref{eq:delay_general} above is only well-defined as long as \(u(s)\neq 0\). However, we will see later in Section~\ref{sec:proof}, that we can extend the vector field also to the case of \(u=0\); it is then just given by \(f(0,v)\). The positive and negative half spaces \(\{u>0\}\) and \(\{u<0\}\) are invariant under the dynamics and separated by the singular line \(\{u=0\}\). The desingularizing function \( z(u,v) = |u|^{-\kappa} \) corresponds to a division by \(u^\kappa\) and \((-u)^\kappa\) on  \(\{u>0\}\) and \(\{u<0\}\) respectively.

 We assume that after the desingularization from \(\tilde{f}\) and \(\tilde{g}\) to \(f\) and \(g\) there is a equilibrium point on the set \(\{u=0\}\), i.e., we have \(f(0,v_*) = 0\) for some \(v_*\in \R\). Without loss of generality we can assume that this equilibrium lies at the origin \((u,v)=(0,0)\). Most importantly, we require that the linearization at this equilibrium is \emph{hyperbolic}, i.e., \(\textnormal{D}f(0,0)\) has no eigenvalue on the imaginary axis, which can always be achieved in the plane~\cite{Dum77} after potentially multiple blow-ups; yet, for large classes of vector fields encountered in practice, a single suitably weighted blow-up suffices to obtain hyperbolicity~\cite{Dumortier1,KuehnBook}. Our main invariant manifold result, stated below in Theorem~\ref{thm:stable_mfld}, will be non-trivial if we suppose that \(\textnormal{D}f(0,0)\) has one positive and one negative eigenvalue as otherwise the stable/unstable manifolds become trivial. In fact, the first key problem for blowing up nonautonomous systems originates from the observation that the equation~\eqref{eq:delay_general} is \emph{not} an ordinary differential equation anymore, but a \emph{delay} equation. The derivative at time \(t\) depends on the whole trajectory on the time interval \([0,t]\). This also means that the delay grows unbounded as time evolves forward.
 Hence, following for example \cite{Hino2002}, we should consider this problem on a space of functions defined on \((-\infty, 0]\). The function containing the data history up to time \(t\) is normalized/shifted onto the domain \((-\infty, 0]\) and forms a so-called segment \((u_t, v_t)\) which is given by \((u_t, v_t)(\theta) = (u, v)(t+\theta)\) for any \(\theta \in (-\infty, 0]\). The right hand side of \eqref{eq:delay_general} is then a nonlinear functional acting on \((u_t, v_t)\). In this notation our equation reads
\begin{equation}\label{eq:maindelay}
	\begin{pmatrix}
		u \\ v 	\end{pmatrix}' = f\big(u_t(0), v_t(0)\big)+ g\big(u_t(0),v_t(0)\big)\, h\!\left(\int_{-t}^{0} \frac{1}{|u_t(s)|^\kappa} ~\txtd s\right)\,.
\end{equation}

Note that equation~\eqref{eq:delay_general} is only solvable in forward time and for \(t \geq 0\). By the use of a time reversal one may pose an analogous problem in backwards time.

To begin with, we need to define a phase space for the delay-equation. Let $\|\cdot\|$ denote the usual Euclidean norm on $\R^2$ and consider first the space
 \[ \tilde{\B} := \left\{ \phi \in  C(\R^-,\R^2) \,,~ \sup_{\theta \in \R^-} \|\phi(\theta) \txte^{\lambda \theta} \| < \infty  \right\} \] for some \(\lambda >0\), which we equip with the norm \(\|\phi\|_{\B} = \sup_{\theta \in \R^-} \txte^{\lambda \theta}\|\phi(\theta)\|  \). In order to give a reasonable meaning to the singular integrals and assure continuity of the vector field, we will further define \(\B \subset \tilde{\B} \), which contains all functions of \(\tilde{\B}\) that are additionally \(M\)-Lipschitz continuous, i.e., that satisfy \(\|\phi(s) -\phi(t) \|\leq M | s-t|\) for all \(s,t \in \R^-\) for a prescribed constant \(M >0\), so
 \[ \B := \left\{ \phi \in  C(\R^-,\R^2)\,,~ \sup_{\theta \in \R^-} \|\phi(\theta) \txte^{\lambda \theta} \| < \infty \,,~ \phi \in \text{Lip}_M\!\left(\R^-,\R^2\right) \right\}\,. \]
 The Lipschitz-constant \(M\) is  taken as the supremum of the vector field of equation~\eqref{eq:maindelay} on a domain \(D\) around \(0\). As \(h\) is bounded and \(f, g\) are regular, this supremum exists on bounded domains and furthermore, by considering a smaller domain, we can assume that \(M < \eta\), with \(\eta\) being the exponential decay rate of \(h\) and \(h'\). Since we are only interested in the local behaviour of the dynamics around the origin, the consideration of only a small neighbourhood of the origin is no restriction. From now on we will always assume that outside this domain the right hand side of~\eqref{eq:maindelay} is suitably modified by a cut-off function. Note that classical solutions of \eqref{eq:delay_general} automatically satisfy this Lipschitz condition. \medskip
 
According to \cite{Mura07} the space \(\tilde{\B}\) is a uniformly fading memory space. Furthermore, the subset of all Lipschitz-continuous functions with a given Lipschitz-constant is closed with respect to the norm \(\|\cdot\|_{\B}\)\,, so \(\B\) is also a uniformly fading memory space.\medskip
 
Given some initial condition \(\phi \in \B\) at initial time \(t=0\) that evolves in forward time according to the differential equation \[X' = \txtD f(0,0)X\]  we define the family of solution operators \(\{V(t)\}_{t\geq 0}\) by

\[ [V(t)\phi](\theta) = \begin{cases}
\phi(t+\theta) ~~~&\text{if} ~~  \theta < -t \,,\\ \txte^{(t+\theta) \txtD f(0,0)} \phi(0) ~~~&\text{if} ~~ \theta \in [-t,0] \,.\end{cases} \]  
\(\{V(t)\}_{t\geq 0}\) forms a strongly continuous semigroup of bounded linear operators on \(\B\). According to the decomposition of \(\R^2 = \tilde{E}^{\txts} \oplus \tilde{E}^{\txtu} \) into the stable and unstable eigenspace of \(\txtD f(0,0)\) we can split the phase space \(\B = E^{\txts} \oplus E^{\txtu}\) into 
\begin{align*}
E^{\txts} &= \{\phi \in \B ,~ \phi(0) \in \tilde{E}^{\txts}\, \} \,,\\
E^{\txtu} &= \{\phi \in \B ,~ \phi(0) \in \tilde{E}^{\txtu} \text{ ~ and~  }  \phi(t) = \txte^{\txtD f(0,0)t}\phi(0) \text{ for all } t \in \R^- \} \,,
\end{align*}
which are invariant under the semigroup \(\{V(t)\}_{t\geq 0}\) since \(\tilde{E}^s\) and \(\tilde{E}^u\) are invariant under \(\txte^{t \txtD f(0,0)}\). Note that with this choice of \(E^u\) we can extend the family of operators \(V(t)\) on \(E^u\) to a group defined for all \(t\in \R\).
We denote by \(\Pi^{\txts}\) and \(\Pi^{\txtu}\) the projection onto \(E^{\txts}\) and \(E^{\txtu}\) respectively. 
According to \cite{Hino2002} there are constants \( C, \alpha> 0\) such that
\begin{align}
\label{eq:exponential_estimate}
\begin{split}
V^{\txts}(t) &:= V(t)\big\vert_{E^{\txts}} \leq C \txte^{-\alpha t} ~~~\text{ for all } t \in \R^+ \text{~~~and } \\
V^{\txtu}(t) &:= V(t)\big\vert_{E^{\txtu}} \leq C \txte^{\alpha t} ~~~~~\text{ for all } t \in \R^- 
\,.
\end{split}
\end{align}

Choosing the exponent \(\lambda\) in the definition of the phase space \(\B\) appropriately we can assure that \(\alpha\) is determined by the eigenvalues of \(\txtD f(0,0)\). Furthermore, we  set \begin{equation}\label{eq:estimateProjectionNorm}
	L := \|\Pi^s\|+\|\Pi^u\|\,.
\end{equation} 

Considering the case that the invariant set \(\{u=0\}\) forms the unstable eigenspace \(\tilde{E}^{\txtu}\), we seek a description of a stable manifold that is transversal to \(\{u=0\}\). As usual, by reversing the order of time (in equation~\eqref{eq:main}, before blowup) one can use the same procedure to obtain unstable manifolds that satisfy the  properties below in backwards time. In the following the expression \(\mu_t(\phi, \sigma)\) denotes the segment at time \(t\) of a solution to the equation~\eqref{eq:maindelay} with initial data \((\phi, \sigma)\in \B\times \R\). 

\begin{theorem}[Stable Manifold]
	\label{thm:stable_mfld}
	Assume that the zero solution of~\(\eqref{eq:maindelay}\) is hyperbolic and the unstable eigenspace \(\tilde{E}^{\txtu}\) of \(Df(0,0)\) is given by \(\{u=0\}\) and let \(\sigma_0 > 0 \). There are \(r,\delta>0\) such that the set	
	\[W_{{\textnormal{loc}}}^{\txts} := \Big\{\, (\phi, \sigma) \in \B \times [\sigma_0, \infty) \,: ~ \|\Pi^{\txts}\phi\|_\B <r, \|\mu_t(\phi, \sigma)\|_\B <\delta ~~\forall~ t \in [\sigma, \infty) \,\Big\} \]   
		
	is a manifold, given locally as the graph of a function \(w\colon B_{E^{\txts}}(r)\times [\sigma_0, \infty) \to E^u\) with the following properties:
	\begin{enumerate}[(i)]
		\item the function \(w\) is Lipschitz-continuous in its first argument, continuous in the second; 
		\item \(w(\cdot, \sigma)\) is differentiable at 0 and \(W_{{\textnormal{loc}}}^{\txts}\) is tangent to \(E^{\txts}\);
		\item for every \(\beta \in (0,\alpha)\) (where \(\alpha\) is the exponential decay-rate of the linear system in equation~\eqref{eq:exponential_estimate}) there is a constant \(N\), such that \[ \|\mu_t(\phi, \sigma)\|_\B \leq  N \txte^{-\beta (t-\sigma)}  \|\phi\|_\B  ~~~\text{ for all  }     (\phi, \sigma) \in W_{{\textnormal{loc}}}^{\txts}\,, ~t \in [\sigma, \infty)\,; \]
		\item \( W_{{\textnormal{loc}}}^{\txts}\) is locally  positively invariant, i.e. for every \((\phi, \sigma) \in W_{{\textnormal{loc}}}^{\txts} \)  the solution \(\mu_t(\phi, \sigma)\) lies in \(W_{{\textnormal{loc}}}^{\txts}\) for every \(t \geq \sigma\) as long as \({\Pi^s\mu_t(\phi, \sigma) \in B_{E^s}(r) \,.}\)	
	\end{enumerate}
\end{theorem}  
\medskip

\emph{Remarks:} \begin{enumerate}[(1)] \vspace{-0.2cm}
\item If one chooses the time \(\tau_0\), which stays fixed under the time rescaling \eqref{eq:rescaling1}, different from~0, the minimal initial time \(\sigma_0\) needs to satisfy \(\sigma_0>\tau_0\) accordingly.
\item We may also extend the stable manifold up to 0 in time-direction by taking $\sigma_0 \to 0$. Note that when doing so the domain of definition of the stable manifold (as obtained by Theorem~\ref{thm:stable_mfld})  shrinks to a point. This is due to the dependence of $r$ and $\delta$ on $\sigma_0$ by means of the function~$\zeta$ defined in Proposition \ref{prop:estimate}.
\end{enumerate} \medskip

The  stable manifold \(W_{{\textnormal{loc}}}^{\txts}\) of equation \eqref{eq:maindelay} constructed above induces an invariant manifold \(\widetilde{W}_{{\textnormal{loc}}}^{\txts}\) in the finite dimensional extended phase space $ \{(u,v, t) \in \R^2 \times \R\}$ of equation \eqref{eq:delay_general} in a natural way: We take some \((\tilde{u}, \tilde{v}) \in \R^2\) and \(\sigma \geq \sigma_0\); then there is \(\phi \in \B\) such that \(\phi(0) = (\tilde{u}, \tilde{v})\) and \(\phi\) is a solution of \eqref{eq:delay_general} on the interval \([-\sigma, 0]\).  Now \((\tilde{u}, \tilde{v}) \) lies on \(\widetilde{W}_{{\textnormal{loc}}}^{\txts}\), if and only if the corresponding history function \(\phi\) lies on \(W_{{\textnormal{loc}}}^{\txts}\). This means the induced manifold \(\widetilde{W}_{{\textnormal{loc}}}^{\txts}\) is obtained by evaluating those elements \(\phi\) of \(W_{{\textnormal{loc}}}^{\txts}\) at \(\theta = 0\), that correspond to classical solutions of \eqref{eq:delay_general}. Note that the \(M\)-Lipschitz continuity is automatically fulfilled due to the choice of \(M\); the behaviour of \(\phi\) before time \(-\sigma\) is neither relevant for the Lyapunov Perron operator and thus the stable manifold, nor can it directly be determined by equation~\eqref{eq:delay_general} as the differential equation
per se is only defined for positive times.  The manifold \(\widetilde{W}_{{\textnormal{loc}}}^{\txts}\) inherits the properties of \(W_{{\textnormal{loc}}}^{\txts}\) pointed out in the theorem above.\medskip

 We next reverse the blowup transformation  \((x,y)= (u,uv)\) and return to our initially considered differential equation
 \begin{equation*}\label{eq:main2}
 \begin{pmatrix}
 x\\ y
 \end{pmatrix}'=F(x,y) +G(x,y) h(\tau)\,.
 \end{equation*} 
 Here one needs to be careful, since we do not only need to change the spatial coordinates but also have a space-dependent time transformation given by
 \[ \rho(t) = \int_{-t}^{0}\frac{1}{|u_t(s)|^\kappa}~\txtd s \,,   \]
 for the segment \(u_t\) of a solution  of \eqref{eq:maindelay} as introduced in \eqref{eq:rescaling2}.  Again we take those $(\phi, \sigma) = (u,v, \sigma) \in W_{{\textnormal{loc}}}^{\txts}$ that correspond to classical solutions. These triplets are now mapped using the blow up map  and the time transformation $\rho$, i.e. we have
 \begin{equation}\label{eq:coordinate_change}
   (u,v,\sigma) ~\mapsto (x,y,\tau)~ = \Big(u(0), u(0)v(0), \int_{-\sigma}^{0}\frac{1}{|u(s)|^\kappa}\, \txtd s \Big) \,. 
\end{equation}
 
 Note that this coordinate transformation is regular as long as $u(s) \neq 0~ \forall s \in [-\sigma, 0]$. Since we only consider $u$ that correspond to classical solutions (of equation \eqref{eq:delay_general}) and the half planes $\{u>0\} $ and $\{u<0\}$ are invariant under \eqref{eq:delay_general}, no problems arise. Clearly the equilibrium solution $(u,v) = (0,0)$ cannot be transformed by this map.\medskip
 
 This transformation yields a locally positively invariant manifold $\mathcal{M}$ in the extended phase space $ \{(x,y, \tau) \in \R^2 \times \R\}$ of equation \eqref{eq:main}. To be more precise we obtain two invariant manifolds $\mathcal{M}^+, \mathcal{M}^-$ in the half spaces $\{x>0\}$ and $\{x<0\}$ as $W_{{\textnormal{loc}}}^{\txts}$ is interrupted by $u=0$, where the change of variables is singular.   Solutions on the transformed manifolds still converge to the origin  (more accurately to the invariant line \(\{x=y=0\} \subset \R^3\)), but due to the rescaling of time an exponential decay can no longer be assured.  \medskip

 Summerizing the previous considerations and applying them not only to the stable, but also the unstable manifold which can obtained by Theorem~\ref{thm:stable_mfld} analogously after reversing the order of time, we obtain the following.
 
 \begin{corollary}
 	There are nonautonomous invariant manifolds of equation \eqref{eq:main} in ${\R^2 \times \R}$  corresponding to the stable and unstable manifolds of equation \eqref{eq:maindelay}. They consist of integral curves of solutions that converge to $\big\{(x,y,\tau) \in \R^2\times \R ~|~ x=y=0 \big\}$ in forward and backward time respectively. The transforming change of variables (in the case of the stable manifold) is given by $~ (u,v,\sigma) \,\mapsto\, (x,y,\tau) \,=\,\Big(u(0), u(0)v(0), \int_{-\sigma}^{0}\tfrac{1}{|u(s)|^\kappa}\, \txtd s \Big) \,.$
 \end{corollary}
 \bigskip

 Note, that $\mathcal{M}$ will not be defined for all initial conditions in a neighbourhood of the origin \((x,y) = (0,0)\) for each particular initial time, since $\int_{-\sigma}^{0}\frac{1}{|u(s)|^\kappa}\, \txtd s \to \infty$ as $u \to 0$.

In order to close this missing part of the domain of the invariant manifold and extend the invariant manifold $\mathcal{M}$ in towards the (singular, for the transformation) line $\{x=y=0\}$ one may take the following approach into consideration. We take some $(u,v, \sigma) \in W_{{\textnormal{loc}}}^{\txts} $ on the stable manifold that corresponds to a classical solution of \eqref{eq:delay_general} and evaluate $(u,v)$ at $ -\sigma$; this is equivalent to taking $(u(0),v(0), \sigma) = (\tilde{u}, \tilde{v}, \sigma) \in \widetilde{W}_{{\textnormal{loc}}}^{\txts}$ and tracing solutions of equation \eqref{eq:delay_general} back until time $t=0$. This yields the set $\widetilde{W}_0$  of those initial conditions \((\hat{u},\hat{v}, 0) \in \R^3\), which are mapped to some  \((\tilde{u}, \tilde{v}, \sigma) \in \widetilde{W}_{{\textnormal{loc}}}^{\txts}\) by the flow of \eqref{eq:delay_general}. However, note that the properties of $\widetilde{W}_0$ are not characterized by Theorem~\ref{thm:stable_mfld}.  

Since \(t=\tau=0\) is the invariant time under the time-reparametrization, changing coordinates in this case becomes only a spatial transformation. Thus, applying the same transformation \eqref{eq:coordinate_change} as before, we obtain a set in \((x,y)\)-coordinates at \(\tau=0\). Solution curves to equation \eqref{eq:main} with respect to these initial conditions  now allow to extend the invariant manifolds towards $\{x=y=0\}$  similarly to backwards-time extension of the invariant manifold $\mathcal{M}$. Properties of $\widetilde{W}_0$ and possible extensions of $\mathcal{M}$  may be considered in a future work. Furthermore, a connection to the stable manifold of the autonomous limit equation $(u,v)' = f(u,v)$  seems conceivable. \bigskip

In summary, we achieved the goal of the blow-up method in the sense that we have employed hyperbolic theory to locally study non-hyperbolic nonautonomous dynamics. The cost is that we have to study a more complicated, even formally infinite-dimensional, dynamical system given by a delay equation with infinite delay.    

\section{Preliminaries}
\label{sec:preliminaries}

The proof of the Stable Manifold Theorem~\ref{thm:stable_mfld} in blown-up space is based on a Lyapunov-Perron method. To set up the method, we need (a generalized version of) the variation of constants formula that is central to defining the Lyapunov-Perron operators, whose fixed point will form the stable manifold. First, let us define an auxiliary function \(\Gamma^n\colon \R^2 \to \B\) by  \[ x \, \mapsto \, \Gamma^nx\,,~~~~\Gamma^n x (\theta) = \begin{cases} (n\theta + 1)x \,~~~~~\text{for } -\frac{1}{n} \leq \theta \leq 0 \,, \\
~~0 \hspace{2cm}\text{for } ~~\theta < -\frac{1}{n}~.
\end{cases}\]
\(\Gamma^n\) is an important technical tool, that allows us to generalize the well known variation of constants formula also to delay-equations  in the phase space \(\B\). Also note, that we have an estimate of the form 
\begin{equation}\label{eq:estimateGamma}
\|\Gamma^nx\|_\B \leq K\|x\| 	
\end{equation}
 for some \(K \in \R\), see also \cite{Hino2002}. Let $A\in\R^{2\times 2}$ denote a matrix, which we shall later on take as $A=\txtD f(0,0)$.  

\begin{proposition}[Variation of Constants Formula, \cite{Hino2002}]\label{prop:VOCformula}
	The segment \(x_t(\phi,\sigma, H)\) of a solution \(x(\cdot, \phi, \sigma, H)\) of the inhomogeneous equation \(x' =Ax + H(t)\) satisfies the equation
	\begin{equation*}
	x_t(\phi,\sigma, H) = V(t-\sigma)\phi + \lim\limits_{n \to \infty} \int_{\sigma}^{t} V(t-s) \Gamma^n H(s) ~\txtd s
	\end{equation*}
	for all \(t \geq \sigma\) as a limit in \(\B\).
	
\end{proposition}

The key steps in the proof of this proposition lie in showing that \[ \lim\limits_{n\to \infty} x_t(\Gamma^nx,0 ,0) = V(t)x  \text{~~~~~~as well as~~~~~}  \lim\limits_{n\to \infty}  \int_{\sigma}^{t}V(t-s) \Gamma^nH(s) ~\txtd s = x_t(0,\sigma, H) \,,\] which then leads to the above formula. \medskip

Another fundamental tool for the proof of Theorem~\ref{thm:stable_mfld} is the following result that connects bounded solutions with their integral representations. It is taken from \cite{Murakami2002} with only very little adaptations and can be proven in the  same way. 

\begin{lemma}[\cite{Murakami2002}]\label{lemma:LemmaJap} \ \\ 	
	Let \(\psi \in E^{\txts}\). If \(y\) is a solution of the equation \(y' = A y + R(y_t,t)\) such that \(\Pi^{\txts}y_{\sigma} = \psi\)\,, \(\sup_{t \geq \sigma}\|y_t\|_\B < \infty\) and \(\sup_{t \geq \sigma}\|R(y_t,t)\|_\B < \infty\), then the segment \(y_t\) satisfies for all \(t \geq \sigma \)\begin{equation*}
	y_t = V(t,\sigma)\psi + \lim\limits_{n \to \infty} \int_{\sigma}^{t} V(t-s) \Pi^{\txts} \Gamma^n R(y_s,s) ~\txtd s - \lim\limits_{n \to \infty} \int_{t}^{\infty} V(t-s) \Pi^{\txtu} \Gamma^n R(y_s,s) ~\txtd s \,. 
	\end{equation*}
	On the other hand, if a function \(y \in C([\sigma, \infty), \B)\) 		
	satisfies 
	\begin{equation*}
	y(t) = V(t,\sigma)\psi + \lim\limits_{n \to \infty} \int_{\sigma}^{t} V(t-s) \Pi^{\txts} \Gamma^n R(y(s),s) ~\txtd s - \lim\limits_{n \to \infty} \int_{t}^{\infty} V(t-s) \Pi^{\txtu} \Gamma^n R(y(s),s) ~\txtd s 
	\end{equation*}
	 for all \(t\in [\sigma, \infty)\), as well as the estimates \(~\sup_{t \geq \sigma}\|y(t)\|_\B < \infty\) and \(\sup_{t \geq \sigma}\|R(y(t),t)\| < \infty\) like above, it defines a solution \(\xi\) given by \[\xi(t) =\begin{cases} [y(t)](0) & \text{if~ } t \geq \sigma, \\ [y(\sigma)](t) & \text{if~ } t < \sigma \end{cases}\] of the functional differential equation~\eqref{eq:main} with \(\Pi^{\txts}\, \xi_\sigma = \psi\) and \(\xi_t = y(t) \).\\
\end{lemma}

\section{Proof}\label{sec:proof}

For the proof of Theorem~\ref{thm:stable_mfld} we will work almost solely with the case \(\kappa = 1\) as this needs the sharpest estimates. If \(\kappa \geq 2\), note that the integral \(\int_{-t}^{0} \frac{1}{|u(s)|^\kappa} ~\txtd s\)  has a lower order singularity as \(u \to 0\), so that the estimates and proofs above carry over to these cases in a suitably adapted version. In fact one can even relax the assumptions of the exponential decay of \(h\) and \(h'\) to be only of polynomial order.
 
As we want to utilize the gained hyperbolic structure of equation~\eqref{eq:maindelay} in comparison to equation~\eqref{eq:main},  we will show in the following that the function \(\cK\colon \B\times\R^+ \to \R^2\) given by 
\begin{equation} \cK(u, v, t) = f(u(0), v(0)) + g(u(0), v(0))h(\int_{-t}^{0}\frac{1}{|u(s)|}~\txtd s)
\label{eq:definitionF}
\end{equation} 
is continuously differentiable in \(\B\)-direction. This allows to decompose \[ \cK(u, v, t) = \txtD f(0,0) (u(0), v(0))^\top + R(u, v, t)\,, \] where \(R \colon \B\times \R^+ \to \R^2\) is a  remaining nonlinearity that satisfies \(R(0,0,t)= \txtD R(0,0,t) = 0 \). However, note that \(\cK\) is only well defined if \(\frac{1}{u}\) is integrable on intervals \([-t, 0]\). Thus, as a first step, we show that there is a continuous extension of \eqref{eq:definitionF} given by \eqref{eq:fulldefinitionF} below defined on the whole space \(\B\). Next, we prove that this function is differentiable with a continuous derivative. Finally, we give an estimate on the error \(\cK(\cdot, \cdot, t)-\txtD \cK(0,0,t)\) between the function \(\cK\) and its derivative \(\txtD \cK\), which will become useful later on. \bigskip

Since we are dealing with continuous functions, as long as \(u\) does not have a zero its inverse \(\frac{1}{u}\) is bounded from above on any interval \([-t, 0]\) for some fixed \(t \in \R^+\). Thus the integral \(\int_{-t}^{0}\frac{1}{|u(s)|}~\txtd s\) is finite and we can define \(\cK\) for those \(u\) by using equation~\eqref{eq:definitionF} directly. For all remaining \((u,v) \in \B\), i.e., when there is a \(s_0 \in \R^-\) such that \(u(s_0) = 0\) we set \(\cK(u, v, t) = f(u(0), v(0)) \) so that \(F\colon \B \times\R^+ \to \R^2\) is defined by~\eqref{eq:fulldefinitionF} below.
Here and in the following  functions/segments, that possess a zero, are called singular, while those which are entirely positive or negative, are called regular functions/segments. 
\begin{proposition}
	For each \(t>0\), the function \(~\cK(\cdot, \cdot, t) \colon \B \to \R^2~\) given by
	\begin{equation}\label{eq:fulldefinitionF} 
	\cK(u, v, t) =
	\begin{cases}
	f(u(0), v(0)) + g(u(0), v(0))h(\int_{-t}^{0}\frac{1}{|u(s)|}~\txtd s) & \text{if \(u\) is regular} \\
	f(u(0), v(0))  & \text{if \(u\) is singular}
	\end{cases}
	\end{equation} is a continuous extension of~\eqref{eq:definitionF}.
\end{proposition}

\begin{proof}
Let \(u \in \B \) be singular and \(s_0 \in \R^-\) such that  \(u(s_0)= 0 \). Since all functions in \(\B\) are \(M\)-Lipschitz continuous, we can estimate \(|u(s)| \leq M |s-s_0| \) and thus \[\int_{-t}^{0} \frac{1}{|u(s)|} ~\txtd s  \,\geq\, \int_{-t}^{0} \frac{1}{M|s-s_0|} ~\txtd s  \,=\, \infty\,.\]

Now consider an approximating sequence \(u_n\) of \(u\) in \(\B\) and assume without loss of generality that \(\|u_n -u\|_{\B}\) is monotonically decreasing.  Note that for all \(s\in [-t,0]\) we have \(|\phi(s)|\leq \txte^{\lambda t} \|\phi\|_{\B} \) for any \(\phi \in \B\). Hence we can estimate \[ |u_n(s)| \leq |u(s)| + \txte^{\lambda t} \|u_n - u\|_{\B} \,,\] which implies  \[ \frac{1}{|u_n(s)|} \geq \frac{1}{|u(s)| + \txte^{\lambda t} \|u_n - u\|_{\B}} \] and by the monotone convergence theorem we can deduce that \[ \lim\limits_{n\to \infty} \int_{-t}^{0} \frac{1}{|u_n(s)|} ~\txtd s  = \int_{-t}^{0} \frac{1}{|u(s)| }~\txtd s  = \infty  \,. \] Consequently, \(h(\int_{-t}^{0} \frac{1}{|u_n(s)|} ~\txtd s ) \to 0\) and since \(f, g\) are continuous as well, the continuity of \(\cK\) at singular arguments follows. 
\end{proof}

Next we tackle the differentiability of \(\cK(\cdot, \cdot, t)\). 

\begin{proposition}\label{prop:cont_diff}
	\(\cK(\cdot, \cdot, t)\) defined  by \eqref{eq:fulldefinitionF} is continuously differentiable
\end{proposition}

\begin{proof}
Using the product and chain rule one easily finds the derivative at regular arguments \((u, v) \in \B\) to be 
\begin{multline}\label{eq:derivative}
[\txtD \cK(u, v, t)] (\phi,\psi) = \txtD f(u(0), v(0))(\phi(0), \psi(0))^\top + \txtD g(u(0), v(0))h\!\left(\int_{-t}^{0} \frac{1}{|u(s)|} ~\txtd s \!\right)(\phi(0), \psi(0))^\top  \\+g(u(0), v(0))h'\!\left(\int_{-t}^{0} \frac{1}{|u(s)|} ~\txtd s \!\right) \int_{-t}^{0} \frac{-\phi(s)}{|u(s)|u(s)} ~\txtd s \,. 
\end{multline}
Furthermore we claim that the derivative at a singular function \(u\) is given by 
\begin{equation*}
[\txtD \cK(u, v, t)] (\phi,\psi) = \txtD f(u(0), v(0))(\phi(0), \psi(0))^\top    .\end{equation*}

In order to show 
\begin{equation}\label{eq:differentiability}
\lim\limits_{(w,z)\to 0}  \frac{1}{\|(w,z)\|_\B}  \big\|  \cK(u+w,v+z,t) - \cK(u,v,t) - Df(u(0), v(0)) (w(0),z(0))^\top   \big\|_\B  = 0   
\end{equation}
we distinguish two cases. If \(u+w\) is a singular function, expression \eqref{eq:differentiability} simplifies to 
\begin{equation*}
\lim\limits_{(w,z)\to 0}  \frac{1}{\|(w,z)\|_\B}  \big\|  f(u(0)+w(0),v(0)+z(0)) - f(u(0),v(0)) - Df(u(0), v(0)) (w(0),z(0))^\top   \big\|_\B  
\end{equation*}
  which vanishes due to the differentiability of \(f\).  
  For the case of a regular function \(u+w\) the limit becomes
  \begin{multline*}
  \lim\limits_{(w,z)\to 0}  \frac{1}{\|(w,z)\|_\B}  \Big\| f(u(0)+w(0), v(0)+z(0)) + g(u(0)+w(0), v(0)+z(0)) h\Big(\int_{-t}^{0}\frac{1}{|u(s)+w(s)|}~\txtd s\Big)\\ - f(u(0),v(0)) - Df(u(0), v(0)) (w(0),z(0))^\top \Big\|_\B \,. 
  \end{multline*}
  Here the first and the last two terms vanish due to the differentiability of \(f\),  so it remains to show that \[ \lim\limits_{(w,z)\to 0}  \frac{1}{\|(w,z)\|_\B} \Big\|g(u(0)+w(0), v(0)+z(0)) h\Big(\int_{-t}^{0}\frac{1}{|u(s)+w(s)|}~\txtd s \Big) \Big\|_\B = 0 \,.\] 
  Note that \(g(u(0)+w(0), v(0)+z(0))\) is uniformly bounded for all sufficiently small \((w,z)\) so we are left with \begin{equation*}
\lim\limits_{\|w\|_{\B} \to 0} \frac{1}{\|(w,z)\|_{\B}} \Big\| h\Big(\int_{-t}^{0} \frac{1}{|u(s)+w(s)|} ~\txtd s\Big) \Big\| = 0 ~.
\end{equation*} 
We use the Lipschitz continuity of \(u\) and \(w\) to estimate \[|u(s)+ w(s)| \leq |w(s_0)| + M |s-s_0|\] (recall that \(u(s_0) = 0\)) and thus obtain \[ \int_{-t}^{0} \frac{1}{|u(s)+w(s)|} ~\txtd s \geq \int_{-t}^{0} \frac{1}{|w(s_0)|+M |s-s_0|}~\txtd s \geq \frac{-1}{M}\ln\Bigg(\frac{|w(s_0)|}{|w(s_0)|+ M\tfrac{t}{2} }\Bigg)  .\] Together with the exponential estimate \eqref{eq:assumption} on \(h\) this yields
\begin{align*}
&\frac{1}{\|(w,z)\|_{\B}} \Big\| h\Big(\int_{-t}^{0} \frac{1}{|u(s)+w(s)|} ~\txtd s\Big)  \Big\| \leq  \frac{1}{w(s_0)\txte^{\lambda s_0}} H \cdot \exp\Bigg(-\eta ~ \int_{-t}^{0} \frac{1}{|u(s)+w(s)|} ~\txtd s\Bigg) \\ &\leq \frac{1}{w(s_0)\txte^{\lambda s_0}} H \cdot \exp\Bigg(-\eta ~\frac{-1}{M}\ln\Bigg(\frac{|w(s_0)|}{|w(s_0)|+M\tfrac{t}{2}}\Bigg)\Bigg) = \frac{H}{w(s_0)\txte^{\lambda s_0}} \Bigg(\frac{|w(s_0)|}{|w(s_0)|+ M\tfrac{t}{2} } \Bigg)^\frac{\eta}{M} ~.
\end{align*}   
Now if \(\|(w,z)\|_{B} \to 0\) also \(w(s_0) \to 0\) and since \(\frac{\eta}{M} > 1\) the last expression converges to 0.

The continuity of the derivative given in \eqref{eq:derivative} at regular functions and in the second argument is apparent, so that again only the singular points are of special interest. As before we consider a singular \(u\) with \(u(s_0) = 0\) as well as a perturbation \(u + w\) thereof. 
It is sufficient to show that \[ h'\Big(\int_{-t}^{0} \frac{1}{|u(s) +w(s)|} ~\txtd s \Big) \int_{-t}^{0} \frac{1}{|u(s) +w(s)|^2} ~\txtd s  \] converges to 0 as \(\|(w,z)\|_{\B} \to 0\), since the other parts are already known to be continuous: \(f \text{ and } g\) by assumption and \(h(...)\) by the arguments below \eqref{eq:fulldefinitionF}.

 Let now \({s_1 := \arg \min_{s\in [-t,0]} |u(s)+w(s)|}\). Note that now, in contrast to the previous situation, \(s_1\) depends on the choice of \(w\). Nevertheless, we can estimate in a similar manner to before 
\begin{equation}\label{eq:logEstimate2}
\int_{-t}^{0} \frac{1}{|u(s)+w(s)|} ~\txtd s \geq \int_{-t}^{0} \frac{1}{|u(s_1)+w(s_1)|+M |s-s_1|}~\txtd s = \frac{-1}{M}\ln\Bigg(\frac{|u(s_1)+w(s_1)|}{|u(s_1)+w(s_1)|+M\tfrac{t}{2}}\Bigg) .
\end{equation}
Furthermore 
\begin{align} \label{eq:estimateHprime}
& h'\Big(\int_{-t}^{0} \frac{1}{|u(s) +w(s)|} ~\txtd s \Big) \int_{-t}^{0} \frac{1} {|u(s) +w(s)|^2} ~\txtd s  \nonumber \\ \leq ~ & h'\Big(\int_{-t}^{0} \frac{1}{|u(s) +w(s)|} ~\txtd s \Big) \int_{-t}^{0} \frac{1}{|u(s) +w(s)|} ~\txtd s \cdot \frac{1}{|u(s_1) +w(s_1)|}  \nonumber \\ \leq ~ & 
H \cdot \exp\Big( -\eta \int_{-t}^{0} \frac{1}{|u(s) +w(s)|} ~\txtd s \Big) \cdot \frac{1}{|u(s_1) +w(s_1)|} 
\end{align}
due to the estimate \eqref{eq:assumption}. Analogous to above we now exploit the estimate \eqref{eq:logEstimate2} to continue
\begin{align*}
& H \cdot \exp\Big( -\eta \int_{-t}^{0} \frac{1}{|u(s) +w(s)|} ~\txtd s \Big) \cdot \frac{1}{|u(s_1) +w(s_1)|} \\ \leq ~ &
H \cdot \exp\Bigg( \frac{\eta}{M}\ln\Bigg(\frac{|u(s_1)+w(s_1)|}{|u(s_1)+w(s_1)|+M\tfrac{t}{2}} \Bigg) \Bigg)  \cdot \frac{1}{|u(s_1) +w(s_1)|} \\ = ~ &
H \cdot  \Bigg(\frac{|u(s_1)+w(s_1)|}{|u(s_1)+w(s_1)|+M\tfrac{t}{2}} \Bigg) ^{\frac{\eta}{M}} \cdot \frac{1}{|u(s_1) +w(s_1)|}.
\end{align*}
Now, as \(\|(w,z)\|_{\B} \to 0\) also \(|w(s_0)| \to 0\) and thus also  
\[|w(s_0)| = |u(s_0)+ w(s_0)| \geq \min_{s\in [-t,0]} |u(s)+ w(s)| = |u(s_1)+ w(s_1)|  \]  converges towards 0. Since \(\frac{\eta}{M} > 1 \), like before  the term \(H \cdot  \Bigg(\frac{|u(s_1)+w(s_1)|}{|u(s_1)+w(s_1)|+M\tfrac{t}{2}} \Bigg) ^{\frac{\eta}{M}} \cdot \frac{1}{|u(s_1) +w(s_1)|}\) is convergent towards 0 as well.
\end{proof}
 
Finally we will give a uniform estimate on the derivative of \(\cK(\cdot, \cdot, t)-\txtD \cK(0,0,t) =: R(\cdot, \cdot, t)\), and therewith the Lipschitz constant of \(R(\cdot, \cdot, t)\),  in a neighbourhood of \(0 \in \B\) for all \(t \in [\sigma_0, \infty)\).

\begin{proposition} \label{prop:estimate}
	Let \(\sigma_0 > 0 \). There is a function \(\zeta \colon  [0,\delta_0] \to \R^+\) for some \(\delta_0 > 0\) with \(\lim_{\delta \to 0} \zeta(\delta) = 0\) such that  \(\|\txtD \cK(u, v, t) - \txtD \cK(0,0,t)\|  \leq  \zeta(\delta)\) for all  \((u,v) \in \B\) with \(\|(u,v)\|_{\B} < \delta\) and all \(t \geq \sigma_0\).  Consequently, \(\zeta(\delta)\) is the Lipschitz constant of \(R(\cdot, \cdot, t)\)  in \(B_\delta(0)\) for all \(t\geq \sigma_0\).
\end{proposition}

\begin{proof}
Considering \eqref{eq:estimateHprime} with the special choice \(u = 0\) we obtain 
\begin{align*}
& h'\Big(\int_{-t}^{0} \frac{1}{|w(s)|} ~\txtd s \Big) \int_{-t}^{0} \frac{1} {|w(s)|^2} ~\txtd s  ~ \leq ~ 
H \cdot \exp\Big( -\eta \int_{-t}^{0} \frac{1}{|w(s)|} ~\txtd s \Big) \cdot \frac{1}{|w(s_1)|} \\  ~\leq~& H \cdot  \Bigg(\frac{|w(s_1)|}{|w(s_1)|+M\tfrac{t}{2}} \Bigg) ^{\frac{\eta}{M}} \cdot \frac{1}{|w(s_1)|} \leq H \cdot  \Bigg(\frac{|w(s_1)|}{|w(s_1)|+M\tfrac{\sigma_0}{2}} \Bigg) ^{\frac{\eta}{M}} \cdot \frac{1}{|w(s_1)|}  =: \gamma(|w(s_1)|)
\end{align*}
Note that there is some \(r_0\) such that \(\gamma(r)\) converges monotonically towards 0 as \(r \to 0\) for all \(r \in [0, r_0]\). Further note that \(\|(w,z)\|_{\B} \geq |w(0)| \geq |w(s_1)|\) so \(\gamma(\|(w,z)\|_{\B})\) gives an upper bound of the term \(h'\Big(\int_{-t}^{0} \frac{1}{|w(s)|} ~\txtd s \Big) \int_{-t}^{0} \frac{1} {|w(s)|^2} ~\txtd s  \). \medskip

We have for all \((u,v) \in \B\) (regular as well as singular functions)
\begin{align*}
&\|\txtD \cK(u, v, t) - \txtD \cK(0,0,t)\|  \leq  \|\txtD f(u(0), v(0)) -\txtD f(0,0)\| \\ &+
\Big\| \txtD g(u(0), v(0))h\Big(\int_{-t}^{0} \frac{1}{|u(s)|} ~\txtd s \Big) \Big\| + \Big\| g(u(0), v(0))h'\Big(\int_{-t}^{0} \frac{1}{|u(s)|} ~\txtd s \Big) \int_{-t}^{0} \frac{1}{|u(s)|^2} ~\txtd s \Big\| 
\end{align*}

Now, the first term converges to 0 as \((u,v)\to 0\) since \(f\) is continuously differentiable and the second term by the estimate on \(h\Big(\int_{-t}^{0} \frac{1}{|u(s)|} ~\txtd s \Big)\) in Proposition~\ref{prop:cont_diff}. Finally, the last term is bounded by a multiple of \(\gamma(\|(u,v)\|_{\B})\) by our considerations above. Consequently, there exists some \(\delta_0 >0\) and a function \(\zeta \colon  [0,\delta_0] \to \R^+\) satisfying \(\zeta(\delta) \to 0\) such that \(\|\txtD \cK(u, v, t) - \txtD \cK(0,0,t)\|  \leq  \zeta(\delta)\) for all \((u,v) \in \B\) with \(\|(u,v)\|_{\B} < \delta\).
\end{proof}\medskip

We now are able to apply the tools of Section~\ref{sec:preliminaries} on our vector field $$\cK(u, v, t) =\txtD \cK(0,0,t) +R(u,v,t) = \txtD f(0,0) (u(0), v(0))^\top + R(u, v, t) $$ and consider Lyapunov-Perron operators. \medskip

\begin{proof}(of Theorem~\ref{thm:stable_mfld})
	
From Proposition~\ref{prop:estimate} we know that \(\|R(u_1, v_1, t) - R(u_2, v_2, t) \| \leq \zeta(\delta) \| (u_1, v_1) -(u_2, v_2)\|_{\B}  \) holds for all \((u_1, v_1), (u_2, v_2) \in B_{\B}(\delta)\) uniformly for all \(t \geq \sigma_0 \), where we use \(B_{A}(r) \) to denote the closed ball of radius \(r\) in the space \(A\) centered at the origin. Since \(R(0,0, t) = 0\) for all \(t\), this especially yields the estimate 
\begin{equation} \label{eq:estimateLipschitz}
	\|R(u, v, t)\| \leq \zeta(\delta) \| (u, v)\|_\B
\end{equation}
 for all \((u, v) \in B_{\B}(\delta) \). \medskip
 
  Now take some \(0 < \beta < \alpha\) and consider the spaces \(Y_\sigma= BC^{\beta}([\sigma,\infty), \B ) \text{ and }\tilde{Y}_\sigma= BC^{\beta}([\sigma, \infty), \tilde{\B})\), where for \(S=\B\) or \(S=\tilde{\B}\) we define the space $$BC^\beta([\sigma,\infty), S) := \{ y \in C([\sigma,\infty), S) \,,~ \sup_{\theta \in [\sigma,\infty)} \txte^{\beta (\theta-\sigma)} \|y(\theta)\|_S  < \infty\}$$ as the Banach space of continuous, exponentially bounded functions, that is endowed with the norm \(\|y\|_{BC^\beta([\sigma,\infty), S)}= \sup_{\theta \in [\sigma,\infty)} \txte^{\beta (\theta-\sigma)} \|y(\theta)\|_S\)\,.\bigskip

For fixed \(r>0\), \(\delta \in [0,\delta_0]\) and every \(\sigma \in [\sigma_0, \infty)\) let us define \(\F_\sigma \colon  B_{E^{\txts}}(r) \times B_{Y_\sigma}(\delta) \to Y_\sigma\) by
\[ \F_\sigma(\psi, y)(t) = V(t-\sigma)\psi + \lim\limits_{n \to \infty}
\int_{\sigma}^{t} V(t-s) \Pi^{\txts} \Gamma^n R(y(s),s) ~\txtd s - \lim\limits_{n \to \infty} \int_{t}^{\infty} V(t-s) \Pi^{\txtu} \Gamma^n R(y(s),s) ~\txtd s  \]
for all \(t \geq \sigma\). We show next, that \(\F_\sigma(\psi, \cdot)\)  is a well defined self-mapping and a contraction for every \(\psi\in B_{E^s}(r) , \sigma\in[\sigma_0, \infty)\) with the aim of applying the contraction mapping theorem/Banach fixed point theorem. \medskip

First, using the estimates  \eqref{eq:exponential_estimate}, \eqref{eq:estimateProjectionNorm}, \eqref{eq:estimateGamma} and \eqref{eq:estimateLipschitz}, we obtain
\begin{multline*}
	\Big\| \lim\limits_{n \to \infty}
	\int_{\sigma}^{t} V(t-s) \Pi^{\txts} \Gamma^n R(y(s),s) ~\txtd s  \Big\|_{\B} \leq C K L\, \zeta(\delta) \int_{\sigma}^{t} \txte^{-\alpha(t-s)} \|y(s)\|_B ~\txtd s \\ \leq  C K L \, \zeta(\delta) \int_{\sigma}^{t} \txte^{-\alpha(t-s)} \|y\|_{Y_\sigma} \txte^{-\beta (s-\sigma)} ~\txtd s \leq  C K L \, \zeta(\delta)\tfrac{1}{\alpha-\beta}\|y\|_{Y_\sigma} \txte^{-\beta (t-\sigma)}
\end{multline*}
as well as 
\begin{multline*}
	\Big\| \lim\limits_{n \to \infty} \int_{t}^{\infty} V(t-s) \Pi^{\txtu} \Gamma^n R(y(s),s) ~\txtd s  \Big\|_{\B} \leq CK L \, \zeta(\delta) \int_{\sigma}^{t} \txte^{-\alpha(t-s)} \|y(s)\|_B ~\txtd s \\\leq C K L \, \zeta(\delta) \int_{t}^{\infty} \txte^{\alpha(t-s)} \|y\|_{Y_\sigma} \txte^{-\beta (s-\sigma)} ~\txtd s  \leq  C K L \, \zeta(\delta)\tfrac{1}{\alpha+\beta}\|y\|_{Y_\sigma} \txte^{-\beta (t-\sigma)}
\end{multline*}

for all \(t \geq \sigma\), which allows to conclude that 
\begin{align*}
	\|\F_\sigma(\psi, y)(t)\|_{\B} &\leq  \left\|V(t-\sigma)\psi\right\|_{\B} +   C K L \, \zeta(\delta)\tfrac{1}{\alpha-\beta}\|y\|_{Y_\sigma} \txte^{-\beta (t-\sigma)} +   C K L \, \zeta(\delta)\tfrac{1}{\alpha+\beta}\|y\|_{Y_\sigma} \txte^{-\beta (t-\sigma)} \\
	&\leq Ce^{\alpha(\sigma-t) } \|\psi\|_\B  +   C K L \, \zeta(\delta)\left(\tfrac{1}{\alpha-\beta} +\tfrac{1}{\alpha+\beta}\right)\|y\|_{Y_\sigma} \txte^{-\beta (t-\sigma)} \\
	&\leq \txte^{-\beta (t-\sigma)} \left( C \txte^{(\alpha-\beta)(\sigma - t)} r  + C K L \, \zeta(\delta)\left(\tfrac{1}{\alpha-\beta} +\tfrac{1}{\alpha+\beta}\right)\delta  \right) \,.
\end{align*}
Consequently we have \begin{equation}\label{eq:bound_F}
	\|\F_\sigma(\psi, y)\|_{Y_\sigma}  =\! \sup_{t \in [\sigma, \infty)} \txte^{-\beta (t-\sigma)}  \|\F_\sigma(\psi, y)\|_{\B} \,\leq\, C \txte^{(\alpha-\beta)(\sigma - t)} r  + C K L \, \zeta(\delta)\left(\tfrac{1}{\alpha-\beta} +\tfrac{1}{\alpha+\beta}\right)\delta \,. \end{equation}

Since \(\zeta(\delta) \to 0\) as \(\delta \to 0\), a sufficiently small choice of \(\delta\) yields \(C K L \, \zeta(\delta)(\tfrac{1}{\alpha-\beta} + \tfrac{1}{\alpha+\beta}) < \frac{1}{2}\,, \)  so that taking \(r<\frac{\delta}{2C}\) implies that  \( \|\F_\sigma(\psi, y)(t)\|_{\B} \leq \delta \). 
Hence \(\F_\sigma\) maps into \(B_{\tilde{Y}_\sigma}(\delta)\).

It remains to verify that the function \(\F_\sigma(\psi, y)(t)(\cdot): \R^- \to X\) is \(M\)-Lipschitz continuous. Since \(\F_\sigma(\psi, y)(t)(\theta) = \F_\sigma(\psi, y)(t+\theta)(0)\) is a solution of the linear inhomogeneous equation \({x'=\txtD f(0,0)x+R(y(t),t)}\) and  the right hand side is bounded by \(M\) in the considered neighbourhood of 0 due to our assumption in Section~\ref{sec:setting_main_result}, the \(M\)-Lipschitz continuity follows directly. \medskip

Finally, proceeding similarly to the estimate~\eqref{eq:bound_F}, we estimate the Lipschitz constant of \(\F_\sigma(\psi, \cdot)\) as follows 
\begin{multline*}
	\| \F_\sigma(\psi, y_1) - \F_\sigma(\psi, y_2) \|_{Y_\sigma} = \sup_{t \in [\sigma, \infty)} \txte^{\beta (t-\sigma)} \|\F(\psi, y_1, \sigma)(t) - \F(\psi, y_1, \sigma)(t) \|_\B \\ \leq \sup_{t \in [\sigma, \infty)} \txte^{\beta (t-\sigma)}  \Big\| \lim\limits_{n \to \infty} \int_{\sigma}^{t} V(t-s) \Pi^{\txts} \Gamma^n \left(R(y_1(s),s)-R(y_2(s),s)\right) ~\txtd s  \Big\|_{\B} \\ + \sup_{t \in [\sigma, \infty)} \txte^{\beta (t-\sigma)}  \Big\| \lim\limits_{n \to \infty} \int_{t}^{\infty} V(t-s) \Pi^{\txtu} \Gamma^n \left( R(y_1(s),s)-R(y_2(s),s) \right) ~\txtd s  \Big\|_{\B}  \\ \leq C K L \, \zeta(\delta)(\tfrac{1}{\alpha-\beta} + \tfrac{1}{\alpha+\beta} ) \|y_1 - y_2\|_{Y_\sigma} ~.
\end{multline*}
Due to our choice of \(\delta\) above, the Lipschitz constant \(C K L \, \zeta(\delta)(\tfrac{1}{\alpha-\beta} + \tfrac{1}{\alpha+\beta} )\)  is less than \(\frac{1}{2}\). The contraction mapping theorem  yields a unique fixed point \(y^*_\sigma(\psi) \) in \(B_{Y_\sigma}(\delta) \), which satisfies $\F_\sigma(\psi, y^*_\sigma(\psi)) = y^*_\sigma(\psi)$  for every \(\psi \in B_{E^s}(r)\) and every \(\sigma \in [\sigma_0,\infty)\). Lemma~\ref{lemma:LemmaJap} yields the correspondence of this fixed point to a unique bounded solution. The map \({w \colon B_{E^s}(r) \times [\sigma_0,\infty) \to E^u}\), the graph of which forms the stable manifold, is therefore given  by \[w(\psi, \sigma) = \Pi^u\, [y^*_\sigma(\psi)](\sigma) = - \lim\limits_{n \to \infty} \int_{\sigma}^{\infty} V(-s) \Pi^{\txtu} \Gamma^n R([y_\sigma^*(\psi)](s), s) ~\text{ds}\] using the just obtained fixed point of \(\F_\sigma(\psi, \cdot)\)\,.\bigskip

Next, let us consider some \(\psi_1, \psi_2 \in B_{E^s}(r)\). We have 
\begin{align*}
	\|y^*_\sigma(\psi_1) - y^*_\sigma(&\psi_2)\|_{Y_\sigma} = \big\|\F_\sigma(\psi_1, y^*_\sigma(\psi_1)) - \F_\sigma(\psi_2, y^*_\sigma(\psi_2))\big\|_{Y_\sigma} \\ &\leq
	\sup_{t \in [\sigma, \infty)} \txte^{\beta (t-\sigma)} \Bigg[
	\left\|V(t-\sigma)(\psi_1 - \psi_2)\right\|_{Y_\sigma} \\ &\quad+  \Big\| \lim\limits_{n \to \infty} \int_{\sigma}^{t} V(t-s) \Pi^{\txts} \Gamma^n \left(R([y^*_\sigma(\psi_1)](s),s)-R([y^*_\sigma(\psi_1)](s),s)\right) ~\txtd s  \Big\|_{Y_\sigma} \\ &\quad+  \Big\| \lim\limits_{n \to \infty} \int_{t}^{\infty} V(t-s) \Pi^{\txtu} \Gamma^n \left( R([y^*_\sigma(\psi_1)](s),s)-R([y^*_\sigma(\psi_1)](s),s) \right) ~\txtd s  \Big\|_{Y_\sigma} \Bigg]\\  
	& \leq \sup_{t \in [\sigma, \infty)} \txte^{\beta (t-\sigma)} C\txte^{-\alpha(t-\sigma)}\|\psi_1 - \psi_2\|_\B + 
	C K L \, \zeta(\delta)\left(\tfrac{1}{\alpha-\beta} +\tfrac{1}{\alpha+\beta}\right)\|y^*_\sigma(\psi_1)y^*_\sigma(\psi_2)\|_{Y_\sigma} \\
	&= ~C\, \|\psi_1 - \psi_2\|_\B +  C K L \, \zeta(\delta)\left(\tfrac{1}{\alpha-\beta} +\tfrac{1}{\alpha+\beta}\right)\|y^*_\sigma(\psi_1)y^*_\sigma(\psi_2)\|_{Y_\sigma} 
\end{align*}
and consequently 
\begin{equation*}
	\|y^*_\sigma(\psi_1) - y^*_\sigma(\psi_2)\|_{Y_\sigma} \leq  C \Big[ 1- C K L \, \zeta(\delta)\left(\tfrac{1}{\alpha-\beta} +\tfrac{1}{\alpha+\beta}\right) \Big]^{-1} \|\psi_1 - \psi_2\|_\B \,.
\end{equation*}
From this we can conclude 
\begin{equation}\label{eq:exponential_decay_of_Y}
	\|[y^*_\sigma(\psi_1)](t) -[ y^*_\sigma(\psi_2)](t)\|_\B \leq \txte^{-\beta (t-\sigma)} C \Big[ 1- C K L \, \zeta(\delta)\left(\tfrac{1}{\alpha-\beta} +\tfrac{1}{\alpha+\beta}\right) \Big]^{-1} \|\psi_1 - \psi_2\|_\B \,,
\end{equation}
which implies Lipschitz continuity of \(w(\psi, \sigma)  = \Pi^u\, [y^*_\sigma(\psi)](\sigma)\) in \(\psi\) upon taking \(t=\sigma\).\bigskip

As the family of Lyapunov-Perron operators \(\mathcal{F}_\sigma\) depends continuously on \(\sigma\), the same holds for \(w\) (one can also view all the operators \(\mathcal{F}_\sigma\) as defined on the same space \(Y_0\) by considering suitably shifted versions). \bigskip

Following a slightly different approach we can also use the implicit function theorem instead of the contraction mapping theorem to show the existence of the stable manifold, very similar to the proof of \cite[Theorem 5]{Mura07}. However, it is hard to find \(t-\)uniform estimates for a suitable neighbourhood using the implicit function theorem; in contrast to our prior strategy via Lyapunov-Perron and the contraction mapping theorem, the implicit function theorem is better suited in an autonomous setting. Yet, it is helpful to study regularity and we will use it here to prove the differentiability at the origin. \medskip

To apply the implicit function theorem, we define \(\G_\sigma \colon  B_{E^{\txts}}(r) \times B_{Y_\sigma}(\delta) \to \tilde{Y}_\sigma\)  by 
\begin{multline*}
\G_\sigma(\psi, y)(t) = y(t)-\F_\sigma(\psi, y) \\ ~~= y(t) - V(t-\sigma)\psi - \lim\limits_{n \to \infty} \int_{\sigma}^{t} V(t-s) \Pi^{\txts} \Gamma^n R(y(s),s) ~\text{ds} + \lim\limits_{n \to \infty} \int_{t}^{\infty} V(t-s) \Pi^{\txtu} \Gamma^n R(y(s),s) ~\txtd s 	
\end{multline*} for all \(t \geq \sigma\). 
Following the proof of \cite[Theorem 5]{Mura07}, adapted to the nonautonomous setting, we obtain by an application of the implicit function theorem  a unique \(C^1\)-mapping \(\Lambda_\sigma: B_{E^{\txts}}(\hat{{r}})  \to B_{Y_\sigma}(\delta)\) with \(\Lambda_\sigma(0)=0\) and \(\G_\sigma(\psi, \Lambda_\sigma(\psi)) =0\) for all  \(\sigma \in [\sigma_0, \infty) \) and all \({\psi \in B_{E^{\txts}}(r)}\), defined in some neighbourhood \(B_{E^{\txts}}(\hat{{r}})\) with an sufficiently small and  \(\sigma\)-dependent \(\hat{r}\) . This means in particular that \[\Lambda_\sigma(\psi)(\sigma) = \psi - \lim\limits_{n \to \infty} \int_{\sigma}^{\infty} V(-s) \Pi^{\txtu} \Gamma^n R(\Lambda_\sigma(\psi)(s),s) ~\txtd s .\]
Defining now \[W(\cdot, \sigma) := \Pi^{\txtu} \,\circ\, \delta_\sigma \,\circ\, \Lambda_\sigma \colon\, B_{E^{\txts}(r)} \to E^{\txtu} , \text{i.e.\,  } W(\psi, \sigma) = - \lim\limits_{n \to \infty} \int_{\sigma}^{\infty} V(-s) \Pi^{\txtu} \Gamma^n R(\Lambda_\sigma(\psi)(s), s) ~\text{ds}\] the stable manifold is also locally given as the graph of \(W\). Due to the uniqueness, \(W\) and \(w\) coincide on their common domain. As a consequence, the function \(w\) that defines the stable manifold is not only continuous but also differentiable at 0. In fact, it is also differentiable in a neighbourhood of 0, but note that this neighbourhood, where one can guarantee differentiability by the implicit function theorem, may shrink to the origin as \(\sigma\) varies.  

Differentiating the relation \(\G_\sigma(\psi, \Lambda_\sigma(\psi)) =0\) we obtain \[ \txtD_\psi\G_\sigma (\psi, \Lambda_\sigma(\psi)) + \txtD_y\G_\sigma(\psi, \Lambda_\sigma(\psi))\txtD \Lambda_\sigma(\psi) =0\,, \] which yields \[ \txtD \Lambda_\sigma(0) = -\txtD_y\G_\sigma(0,0) = V(\cdot)\,. \] From this we can conclude that \(\txtD_\psi W(0, \sigma) = [\Pi^{\txtu} \,\circ\, \delta_\sigma \,\circ\, \txtD \Lambda_\sigma](0) = \Pi^{\txtu} V(0) = \Pi^{\txtu} \), which vanishes when evaluated on \(E^{\txts}\). This shows that the manifold given by the graph of \(W\) is tangential to \(E^{\txts}\) at the origin.\medskip

Next, consider some initial condition \((\phi, \sigma) = (\psi + w(\psi, \sigma), \sigma) \in W_{{\textnormal{loc}}}^{\txts}\) with corresponding projection \(\psi = \Pi^s\phi \in B_{E^s}(r)\). Using estimate \eqref{eq:exponential_decay_of_Y} (and since \(y^*_\sigma(0)(t) = 0\)), the segment \( \mu_t(\phi, \sigma) = [y^*_\sigma(\psi)](t)\) to the initial data \((\phi, \sigma)\) satisfies  
\begin{multline*}
\|\mu_t(\phi, \sigma)\|_\B = \|[y^*_\sigma(\psi)](t)\|_\B \,\leq\, \txte^{-\beta (t-\sigma)} C \Big[ 1- C K L \, \zeta(\delta)\left(\tfrac{1}{\alpha-\beta} +\tfrac{1}{\alpha+\beta}\right) \Big]^{-1} \|\psi\|_\B  \\ \leq\, N  \txte^{-\beta (t-\sigma)} \|\phi\|_\B ~
\end{multline*}
for any \(t\geq \sigma\) with \(N:= CL \Big[ 1- C K L \, \zeta(\delta)\left(\tfrac{1}{\alpha-\beta} +\tfrac{1}{\alpha+\beta}\right) \Big]^{-1}\!,\) proving assertion~\((iii)\). \medskip

Positive invariance of the stable manifold \(W_{{\textnormal{loc}}}^{\txts}\) follows directly from its definition. 

\end{proof}
\medskip

\section{Conclusion and Outlook}
\label{sec:outlook}
In this work, we have provided a blow-up method to desingularize equilibrium points for nonautonomous differential equations in the plane. We have focused on the key challenge to understand the structure of the blown-up desingularized system, which turned out to be a delay equation with infinite delay. For this system, the time reparametrization required us to study regularity properties near the equilibrium to be able to employ Lyapunov-Perron and implicit function techniques to obtain invariant stable/unstable manifolds and their corresponding finite dimensional counterparts. One may interpret our results as establishing a rigorous link between local non-hyperbolic behaviour of nonautonomous systems to hyperbolic behaviour of dynamics with memory. From a geometric viewpoint, stable/unstable manifolds of non-hyperbolic nonautonomous system can now be viewed as projections of families of infinite-dimensional stable/unstable manifolds.\medskip

Having a blow-up method available for nonautonomous dynamics, opens up the opportunity to tackle several problems; we just mention a few possible directions here. (1) It may happen that a single desingularization step does not yield a hyperbolic equilibrium point but produces a collection of new non-hyperbolic equilibria, so that the same procedure may applied again. As Dumortier has shown in \cite{Dum77} a finite number of repeated blowups is sufficient to fully desingularize the dynamics if a planar vector field satisfies a Lojasievicz inequality. Hence, one should also study repeated blow-ups in our case. (2) There is space left for generalization of the method. Besides more general asymptotically autonomous equations, one could also look at the case of different limits in forward and backward time or periodically forced equations, just to name a few examples where Lyapunov exponent can be calculated with little effort. (3) It is often a natural step from nonautonomous systems via skew-product flows to better understand stochastically forced dynamics, where the lack of local hyperbolicity also occurs frequently. Therefore, a major challenge would be to incorporate time-dependent stochastic terms into a blow-up method.\bigskip

\textbf{Acknowledgments:} This work was supported by the German Research Foundation (DFG) via the SFB-TRR109 ``Discretization in Geometry and Dynamics''. CK acknowledges also partial support via a Lichtenberg Professorship of the VolkswagenFoundation. LA acknowledges support from the graduate program TopMath of the Elite Network of Bavaria and the TopMath Graduate Center of TUM Graduate School.

\bibliographystyle{plain}
\bibliography{BibNonautoBlowup}

\begin{thebibliography}{10}

\bibitem{Chicone2006}
C.~Chicone.
\newblock {\em Ordinary Differential Equations with Applications}, volume~34 of
  {\em Texts in Applied Mathematics}.
\newblock {Springer Science+Business Media Inc}, New York, NY, 2 edition, 2006.

\bibitem{Diekmannetal}
O.~Diekmann, S.A. van Gils, S.M.~Verduyn Lunel, and H.-O. Walther.
\newblock {\em Delay Equations: functional-, complex- and nonlinear analysis}.
\newblock Springer, 1995.

\bibitem{Dum77}
F.~Dumortier.
\newblock Singularities of vector fields on the plane.
\newblock {\em Journal of Differential Equations}, 23(1):53--106, 1977.

\bibitem{Dumortier}
F.~Dumortier.
\newblock {\em Singularities of Vector Fields}.
\newblock IMPA, Rio de Janeiro, Brazil, 1978.

\bibitem{Dumortier1}
F.~Dumortier.
\newblock Techniques in the theory of local bifurcations: Blow-up, normal
  forms, nilpotent bifurcations, singular perturbations.
\newblock In D.~Schlomiuk, editor, {\em Bifurcations and Periodic Orbits of
  Vector Fields}, pages 19--73. Kluwer, Dortrecht, The Netherlands, 1993.

\bibitem{DumortierRoussarie}
F.~Dumortier and R.~Roussarie.
\newblock {\em Canard Cycles and Center Manifolds}, volume 121 of {\em Memoirs
  Amer. Math. Soc.}
\newblock AMS, 1996.

\bibitem{GH}
J.~Guckenheimer and P.~Holmes.
\newblock {\em Nonlinear Oscillations, Dynamical Systems, and Bifurcations of
  Vector Fields}.
\newblock Springer, New York, NY, 1983.

\bibitem{Hale3}
J.K. Hale.
\newblock {\em Theory of Functional Differential Equations}.
\newblock Springer, 1977.

\bibitem{Hino2002}
Y.~Hino, S.~Murakami, T.~Naito, and N.~{van Minh}.
\newblock A variation-of-constants formula for abstract functional differential
  equations in the phase space.
\newblock {\em Journal of Differential Equations}, 179(1):336--355, 2002.

\bibitem{Hironaka1}
H.~Hironaka.
\newblock Resolution of singularities of an algebraic variety over a field of
  characteristic zero: {I}.
\newblock {\em Ann. of Math.}, 79(1):109--203, 1964.

\bibitem{Hironaka2}
H.~Hironaka.
\newblock Resolution of singularities of an algebraic variety over a field of
  characteristic zero: {II}.
\newblock {\em Ann. of Math.}, 79(2):205--326, 1964.

\bibitem{KloedenRasmussen}
P.E. Kloeden and M.~Rasmussen.
\newblock {\em Nonautonomous Dynamical Systems}.
\newblock AMS, 2011.

\bibitem{Krupa2001}
M.~Krupa and P.~Szmolyan.
\newblock Extending geometric singular perturbation theory to nonhyperbolic
  points---fold and canard points in two dimensions.
\newblock {\em SIAM Journal on Mathematical Analysis}, 33(2):286--314, 2001.

\bibitem{KuehnBook}
C.~Kuehn.
\newblock {\em Multiple Time Scale Dynamics}.
\newblock Springer, 2015.

\bibitem{Morel2008}
J.-M. Morel, F.~Takens, B.~Teissier, Luis Barreira, and Claudia Valls.
\newblock {\em Stability of Nonautonomous Differential Equations}, volume 1926.
\newblock {Springer Berlin Heidelberg}, Berlin, Heidelberg, 2008.

\bibitem{Murakami2002}
S.~Murakami and N.~Minh.
\newblock Some invariant manifolds for abstract functional differential
  equations and linearized stabilities.
\newblock {\em Vietnam Journal of Mathematics}, 30, 2002.

\bibitem{Mura07}
S.~Murakami and Y.~Nagabuchi.
\newblock Invariant manifolds for abstract functional differential equations
  and related volterra difference equations in a banach space.
\newblock {\em Funkcialaj Ekvacioj}, 50(1):133--170, 2007.

\bibitem{Panazzolo2}
D.~Panazzolo.
\newblock Resolution of singularities of real-analytic vector fields in
  dimension three.
\newblock {\em Acta Math.}, 197:167--289, 2006.

\bibitem{Perron1929}
O.~Perron.
\newblock {Ü}ber {S}tabilität und asymptotisches {V}erhalten der {I}ntegrale
  von {D}ifferentialgleichungssystemen.
\newblock {\em Mathematische Zeitschrift}, 29(1):129--160, 1929.

\bibitem{Poetzsche2}
C.~P{\"o}tzsche.
\newblock {\em Geometric Theory of Discrete Nonautonomous Dynamical Systems}.
\newblock Lecture Notes in Mathematics. Springer, 2010.

\end{thebibliography}

 \end{document}